\documentclass[graybox]{svmult}

\usepackage{type1cm}
\usepackage{makeidx}
\usepackage{graphicx}
\usepackage{multicol}
\usepackage[bottom]{footmisc}
\usepackage{newtxtext}
\usepackage[varvw]{newtxmath}

\makeindex

\begin{document}

\title*{A Survey on the Div-Curl Lemma and Some Extensions to Fractional Sobolev Spaces}
\titlerunning{A Survey on the Div-Curl Lemma} 
\author{Maicol Caponi\orcidID{0000-0002-0932-0022}}
\institute{Maicol Caponi\at Dipartimento di Ingegneria e Scienze dell'Informazione e Matematica,
Università degli Studi dell'Aquila,
Via Vetoio 1, Coppito, 67100 L'Aquila, Italy,
\email{maicol.caponi@univaq.it}}

%---------------------------
%   Title
%---------------------------

\maketitle

%---------------------------
%   Abstract
%---------------------------

\abstract*{This chapter recalls the classical formulation of the Div-Curl lemma along with its proof, and presents some possible generalizations in the fractional setting, within the framework of the Riesz fractional gradient and divergence introduced by Shieh and Spector (2015) and further developed by Comi and Stefani (2019).}

\abstract{This chapter recalls the classical formulation of the Div–Curl lemma along with its proof, and presents some possible generalizations in the fractional setting, within the framework of the Riesz fractional gradient and divergence introduced by Shieh and Spector (2015) and further developed by Comi and Stefani (2019).}

%---------------------------
%   Introduction
%---------------------------

\section{Introduction}

Partial differential equations arising in continuum mechanics are often nonlinear, which creates significant analytical challenges. One of the main difficulties is that weak convergence, commonly used in variational problems, does not behave well with respect to nonlinear operations. In particular, the product of two weakly converging sequences does not necessarily converge to the product of their limits.

The theory of {\em Compensated Compactness}, developed by Murat and Tartar in the late 1970s (see~\cite{Murat78,Murat81,Tartar78,Tartar79,Tartar83}), provides a powerful framework to overcome this problem. Its starting point, the {\em Div–Curl lemma}, shows that the scalar product of two weakly converging sequences can itself converge weakly, under suitable conditions on their divergence and curl. More generally, the key idea of {\em Compensated Compactness} is that certain nonlinear expressions may exhibit weak continuity when appropriate differential constraints are imposed.

Originally introduced in the context of {\em homogenization}, the {\em Div–Curl lemma} was first used to prove compactness results for nonlinear differential operators with oscillating coefficients. It soon found applications in other areas, such as {\em conservation laws}, where it helps to derive entropy solutions via Young measures, and {\em nonlinear elasticity}, where it can be used to show the weak continuity of the determinant of the Jacobian matrix.

The aim of this chapter is twofold. First, in Section~\ref{sec:1} we recall the classical formulation of the Div–Curl lemma and present its proof, following the approach in~\cite{Evans91} (see also~\cite{Murat78}), which is based on Rellich’s compactness theorem and standard local elliptic regularity estimates. Second, in Section~\ref{sec:2} we discuss several extensions of the Div–Curl lemma to the fractional setting of~\cite{CS19,SS15}. These nonlocal versions may be useful for analyzing the compactness properties of fractional differential operators with oscillating coefficients, and could find applications in the study of nonlocal partial differential equations.

%---------------------------
%   Div-Curl lemma
%---------------------------

\section{The Div-Curl Lemma}\label{sec:1}

In this section, we introduce the notation that will be used throughout the chapter. We then recall the classical formulation of the Div–Curl lemma and provide a complete proof, based on Rellich’s compactness theorem and standard elliptic regularity estimates. This argument will serve as a reference point for the nonlocal generalizations discussed in the next section.

Let $n\ge 2$, let $\Omega\subset\mathbb{R}^n$ be an open set, and let $p\in(1,\infty)$ be fixed. We denote by $p'\coloneq\frac{p}{p-1}\in(1,\infty)$ the H\"older conjugate exponent of $p$. Given two open sets $A,B\subseteq\mathbb{R}^n$, we write $A\Subset B$ if there exists a compact set $K\subset\mathbb{R}^n$ such that $A\subset K\subset B$. The space of distributions on $\Omega$, namely the continuous dual space of $C_c^\infty(\Omega)$ endowed with the strong dual topology, is denoted by $\mathcal{D}'(\Omega)$. 

The Lebesgue and Sobolev spaces are defined as usual. In particular, we consider
\begin{equation*}
W^{1,p}_0(\Omega)\coloneq\overline{C_c^\infty(\Omega)}^{\|\,\cdot\,\|_{W^{1,p}(\mathbb{R}^n)}},\qquad W^{-1,p}(\Omega)\coloneq (W^{1,p'}_0(\Omega))', 
\end{equation*}
as well as the local Sobolev spaces
\begin{align*}
W^{1,p}_{\rm loc}(\Omega)&\coloneq\{u\in L^p_{\rm loc}(\Omega)\,:\,\nabla u\in L^p_{\rm loc}(\Omega;\mathbb{R}^n)\},\\
W^{-1,p}_{\rm loc}(\Omega)&\coloneq\{f\in\mathcal{D}'(\Omega)\,:\,f\in W^{-1,p}(\Omega')\text{ for all open sets $\Omega'\Subset\Omega$}\}.
\end{align*}

\begin{remark}\label{rem:negative-sob}
Given a function $u\in L^p_{\rm loc}(\Omega)$, then both $u$ and its distributional derivatives $\partial_i u$ (in the sense of distributions) belong to $W^{-1,p}_{\rm loc}(\Omega)$ for all $i\in\{1,\dots, n\}$. In particular, for all bounded open sets $\Omega'\Subset\Omega$, we have
\begin{align*}
\langle u,v\rangle_{W^{-1,p}(\Omega')\times W^{1,p'}_0(\Omega')} &=\int_{\Omega'}uv\,{\rm d}x& &\text{for all $v\in W^{1,p'}_0(\Omega')$},\\
\langle\partial_i u,v\rangle_{W^{-1,p}(\Omega')\times W^{1,p'}_0(\Omega')} &=-\int_{\Omega'}u\partial_iv\,{\rm d}x& &\text{for all $v\in W^{1,p'}_0(\Omega')$}.
\end{align*}
These identities imply the following estimates
\begin{equation*}
\|u\|_{W^{-1,p}(\Omega')}\le\|u\|_{L^p(\Omega')},\qquad\|\partial_iu\|_{W^{-1,p}(\Omega')}\le\|u\|_{L^p(\Omega')}.
\end{equation*}
\end{remark}

\begin{definition}\label{def:curl}
Given a vector field $F\in L^p_{\rm loc}(\Omega;\mathbb{R}^n)$, we define $\operatorname{div}F\in W^{-1,p}_{\rm loc}(\Omega)$ and $\operatorname{curl}F\in W^{-1,p}_{\rm loc}(\Omega,\mathbb{R}^{n\times n})$, respectively, by
\begin{align*}
\operatorname{div}F\coloneq\sum_{j=1}^n\partial_j F_j,\qquad
(\operatorname{curl}F)_{ij}\coloneq\partial_j F_i-\partial_i F_j\quad\text{for all $i,j\in\{1,\dots,n\}$}.
\end{align*}
Similarly, given a matrix field ${\bf A}\in L^p_{\rm loc}(\Omega;\mathbb{R}^{n\times n})$, we define $\operatorname{div} {\bf A}\in W^{-1,p}_{\rm loc}(\Omega;\mathbb{R}^n)$ by
\begin{equation*}
(\operatorname{div}{\bf A})_i\coloneq\sum_{j=1}^n\partial_j {\bf A}_{ij}\quad\text{for all $i\in\{1,\dots,n\}$}.
\end{equation*}
\end{definition}

We are now ready to state the \emph{Div-Curl lemma}.

\begin{theorem}[Div-Curl lemma]\label{thm:divcurl}
Let $(F^k)_k\subset L^p_{\rm loc}(\Omega;\mathbb{R}^n)$, $F^\infty\in L^p_{\rm loc}(\Omega;\mathbb{R}^n)$, $(G^k)_k\subset L^{p'}_{\rm loc}(\Omega;\mathbb{R}^n)$, and $G^\infty \in L^{p'}_{\rm loc}(\Omega;\mathbb{R}^n)$ satisfy as $k\to\infty$
\begin{align}\label{eq:weakconv}
F^k\to F^\infty\quad\text{weakly in $L^p_{\rm loc}(\Omega;\mathbb{R}^n)$},\qquad 
G^k\to G^\infty\quad\text{weakly in $L^{p'}_{\rm loc}(\Omega;\mathbb{R}^n)$}.
\end{align}
Assume that as $k\to\infty$
\begin{align}
\operatorname{div} F^k&\to\operatorname{div} F^\infty & &\text{strongly in $W^{-1,p}_{\rm loc}(\Omega)$},\label{eq:divvk}\\
\operatorname{curl} G^k&\to\operatorname{curl} G^\infty& &\text{strongly in $W^{-1,p'}_{\rm loc}(\Omega;\mathbb{R}^{n\times n})$}.\label{eq:curlwk}
\end{align}
Then, as $k\to\infty$
\begin{equation}\label{eq:divcurl}
F^k\cdot G^k\to F^\infty\cdot G^\infty\quad\text{weakly* in $\mathcal{D}'(\Omega)$}.
\end{equation}
\end{theorem}

\begin{remark}
Let us briefly comment on Theorem~\ref{thm:divcurl}.
\begin{enumerate}
\item[(a)] The Div-Curl lemma is sometimes stated by requiring that the sequences $(\operatorname{div} F^k)_k$ and $(\operatorname{curl} G^k)_k$ are uniformly bounded in $L^p_{\rm loc}(\Omega)$ and $L^{p'}_{\rm loc}(\Omega;\mathbb{R}^{n\times n})$, respectively. Due to the compact embedding $L^p_{\rm loc}(\Omega)\subset W^{-1,p}_{\rm loc}(\Omega)$, these uniform bounds imply the strong convergence assumptions~\eqref{eq:divvk} and~\eqref{eq:curlwk}. In this case, the Div-Curl lemma can be interpreted as follows: if two sequences converge weakly, and if a suitable combination of their derivatives (namely, their divergence and curl) remain bounded, then the nonlinear expression given by their scalar product converges, in the sense of distributions, to the product of the weak limits.

\item[(b)] In the case $p=2$, Theorem~\ref{thm:divcurl} can be proved in a simple way by using the Fourier transform (this was, indeed, the original approach of Murat and Tartar). The interested reader can find this argument, for instance, in~\cite[Lemma~7.2]{Tartar09}.

\item[(c)] Under the assumptions of Theorem~\ref{thm:divcurl}, we deduce that the sequence $(F^k\cdot G^k)_k$ is uniformly bounded in $L^1_{\rm loc}(\Omega)$. In particular, $(F^k\cdot G^k)_k$ is uniformly bounded in the space of Radon measures $\mathcal{M}(\Omega)\coloneq (C_c(\Omega))'$. Therefore, by the Banach-Alaoglu's theorem, we conclude that, as $k\to\infty$
\begin{equation*}
F^k\cdot G^k\to F^\infty\cdot G^\infty\quad\text{weakly* in }\mathcal{M}(\Omega).
\end{equation*}
One may hope to improve this convergence, for instance to have as $k\to\infty$
\begin{equation*}
F^k\cdot G^k\to F^\infty\cdot G^\infty\quad\text{weakly in } L^1_{\rm loc}(\Omega).
\end{equation*}
However, this is generally false, as shown in~\cite[Lemma~7.3]{Tartar09}. In particular, one can construct two sequences $(F^k)_k, (G^k)_k\subset L^2(B_1;\mathbb{R}^n)$, where $B_1\subset\mathbb{R}^n$ is the unit ball, such that as $k\to\infty$
\begin{align*}
&F^k \to 0 & &\text{weakly in $L^2(B_1;\mathbb{R}^n)$}, &
&G^k \to 0 & &\text{weakly in $L^2(B_1;\mathbb{R}^n)$},\\
&\operatorname{div} F^k = 0 & &\text{in $W^{-1,2}(B_1)$}, &
&\operatorname{curl} G^k = 0 & &\text{in $W^{-1,2}(B_1;\mathbb{R}^{n\times n})$},
\end{align*}
but as $k\to\infty$
\begin{equation*}
F^k\cdot G^k\not\to 0\quad\text{weakly in $L^1(B_1)$}.
\end{equation*}
\item[(d)] A generalization of Theorem~\ref{thm:divcurl}, where $p'$ is replaced by an exponent $q\in(1,\infty)$ satisfying $1\le 1/p+ 1/q\le 1+ 1/n$, can be found in~\cite[Theorem~2.3]{BCDM09}. Another generalization of Theorem~\ref{thm:divcurl}, which assumes weaker conditions on $(\operatorname{div} F^k)_k$ and $(\operatorname{curl}G^k)_k$, but requires the equi-integrability of the sequence $(F^k\cdot G^k)_k$, is presented in~\cite{CDM11}. Finally, for readers interested in the setting of Carnot groups, an analogue of Theorem~\ref{thm:divcurl} is proved in~\cite[Theorem~5.1]{BFTT10}.
\end{enumerate}
\end{remark}

To gain intuition about why this result holds, let us consider a simplified case where the vector fields $G^k$ and $G^\infty$ admit potentials, i.e., there exist scalar functions $w^k$ and $w^\infty$ such that $G^k =\nabla w^k$ and $G^\infty =\nabla w^\infty$. Then, $\operatorname{curl} G^k =\operatorname{curl} G^\infty = 0$, so that~\eqref{eq:curlwk} trivially holds. In this case, the Div-Curl lemma can be proved more directly as a consequence of the integration by parts formula together with Rellich's compactness theorem.

\begin{theorem}[Simplified Div-Curl lemma]\label{thm:simple-divcurl}
Let $(F^k)_k\!\subset\! L^p_{\rm loc}(\Omega;\mathbb{R}^n)$, $F^\infty\in L^p_{\rm loc}(\Omega;\mathbb{R}^n)$, $(w^k)_k\subset  W^{1,p'}_{\rm loc}(\Omega)$, and $w^\infty \in W^{1,p'}_{\rm loc}(\Omega)$ satisfy as $k\to\infty$
\begin{equation}\label{eq:simple-weakconv}
F^k\to F^\infty\quad\text{weakly in $L^p_{\rm loc}(\Omega;\mathbb{R}^n)$},\qquad 
w^k\to w^\infty\quad\text{weakly in $W^{1,p'}_{\rm loc}(\Omega)$}.
\end{equation}
Assume that as $k\to\infty$
\begin{equation}\label{eq:simple-divvk}
\operatorname{div} F^k\to\operatorname{div} F^\infty\quad\text{strongly in $W^{-1,p}_{\rm loc}(\Omega)$}.
\end{equation}
Then, as $k\to\infty$
\begin{equation}
F^k\cdot\nabla w^k\to F^\infty\cdot\nabla w^\infty\quad\text{weakly* in $\mathcal{D}'(\Omega)$}.\label{eq:svkzk}
\end{equation}
\end{theorem}

\begin{proof}
We fix a function $\varphi\in C_c^\infty(\Omega)$ and a bounded open set $\Omega'\subset\mathbb{R}^n$ with smooth boundary such that $\operatorname{supp}\varphi\subset\Omega'\Subset\Omega$.

For all $k\in\mathbb{N}$ we have
\begin{align*}
\int_{\Omega}\varphi F^k\cdot\nabla w^k\,{\mathrm d}x&=\int_{\Omega'} F^k\cdot\nabla (\varphi w^k)\,{\mathrm d}x-\int_{\Omega'} F^k\cdot\nabla\varphi w^k\,{\mathrm d}x\\
&=-\langle\operatorname{div} F^k,\varphi w^k\rangle_{W^{-1,p}(\Omega')\times W^{1,p'}_0(\Omega')}-\int_{\Omega'} F^k\cdot\nabla\varphi w^k\,{\mathrm d}x.
\end{align*}
Thanks to~\eqref{eq:simple-weakconv} and Rellich's theorem, as $k\to\infty$ we derive 
\begin{align}\label{eq:simple-weakconv2}
\varphi w^k\to\varphi w^\infty\quad\text{weakly in $W^{1,p'}_0(\Omega')$},\qquad
w^k\to w^\infty\quad\text{strongly in $L^{p'}(\Omega')$}.
\end{align}
By combining~\eqref{eq:simple-weakconv}--\eqref{eq:simple-divvk} with~\eqref{eq:simple-weakconv2}, we get
\begin{align*}
\lim_{k\to\infty}\langle\operatorname{div} F^k,\varphi w^k\rangle_{W^{-1,p}(\Omega')\times W^{1,p'}_0(\Omega')}&=\langle\operatorname{div} F^\infty,\varphi w^\infty\rangle_{W^{-1,p}(\Omega')\times W^{1,p'}_0(\Omega')},\\
\lim_{k\to\infty}\int_{\Omega'} F^k\cdot\nabla\varphi w^k\,{\mathrm d}x&=\int_{\Omega'} F^\infty\cdot\nabla\varphi w^\infty\,{\mathrm d}x.
\end{align*}
Therefore,
\begin{align*}
\lim_{k\to\infty}\int_{\Omega}\varphi F^k\cdot\nabla w^k\,{\mathrm d}x&=-\langle\operatorname{div} F^\infty,\varphi w^\infty\rangle_{W^{-1,p}(\Omega')\times W^{1,p'}_0(\Omega')}+\int_{\Omega'} F^\infty\cdot\nabla\varphi w^\infty\,{\mathrm d}x\\
&=\int_{\Omega'} F^\infty\cdot\nabla (\varphi w^\infty)\,{\mathrm d}x-\int_{\Omega'} F^\infty\cdot\nabla\varphi w^\infty\,{\mathrm d}x\\
&=\int_{\Omega}\varphi F^\infty\cdot\nabla w^\infty\,{\mathrm d}x,
\end{align*}
which gives~\eqref{eq:svkzk}.
\end{proof}

\begin{remark}
The simplified formulation of the Div–Curl lemma in Theorem~\ref{thm:simple-divcurl} is sufficient for some applications, for example, in proving the compactness of nonlinear differential operators with oscillating coefficients, or the weak continuity of the determinant of the Jacobian matrix. One advantage of this version is that it avoids the explicit use of the curl operator, making it adaptable to non-Euclidean settings more general than Carnot groups (see~\cite[Theorem~3.1]{MPV23}).
\end{remark} 

Let us return to the general form of the Div-Curl lemma in Theorem~\ref{thm:divcurl}. The key idea is to apply a Helmholtz decomposition to the sequence $(G^k)_k$ and its limit $G^\infty$, writing
\begin{equation*}
G^k=H^k+\nabla w^k\quad\text{for all $k\in\mathbb{N}$},\qquad G^\infty = H^\infty +\nabla w^\infty,
\end{equation*}
where, as $k\to\infty$,
\begin{equation*}
H^k\to H^\infty\quad\text{strongly in $L^{p'}_{\rm loc}(\Omega;\mathbb{R}^n)$},\qquad w^k\to w^\infty\quad\text{weakly in $W^{1,p'}_{\rm loc}(\Omega)$}.
\end{equation*}
To construct such a decomposition, we recall some local regularity results for the distributional solutions $u\in L^1_{\rm loc}(\Omega)$ to the Poisson equation
\begin{equation*}
-\Delta u = f\quad\text{in }\mathcal{D}'(\Omega),
\end{equation*}
where the right-hand side $f$ belongs either to $W^{-1,p}_{\rm loc}(\Omega)$ or to $L^p_{\rm loc}(\Omega)$. These results are quite classical (see, e.g.,~\cite{Evans10}), but for the reader's convenience we provide a simple proof following a new strategy introduced in~\cite{DFF20}. The key ingredient is the following proposition, which can be interpreted as a $W^{1,p}$-version of the classical Riesz representation theorem.

\begin{proposition}[{\cite[Theorem~1.2, Chapter~2]{SS96}}]\label{prop:priesz}
Let $\Omega\subset\mathbb{R}^n$ be a bounded open set with boundary $\partial\Omega$ of class $C^1$. For every $f\in W^{-1,p}(\Omega)$, there exists a unique function $v\in W^{1,p}_0(\Omega)$ satisfying
\begin{equation*}
\int_\Omega\nabla v\cdot\nabla\phi\,{\rm d}x =\langle f,\phi\rangle_{W^{-1,p}(\Omega)\times W^{1,p'}_0(\Omega)}\quad\text{for all $\phi\in W^{1,p'}_0(\Omega)$}.
\end{equation*}
Moreover, there exists a constant $C=C(n,p,\Omega)>0$ such that
\begin{equation}\label{eq:reg-est}
\|v\|_{W^{1,p}(\Omega)}\le C\|f\|_{W^{-1,p}(\Omega)}.
\end{equation}
\end{proposition}

\begin{remark}
Actually, in~\cite[Theorem~1.2, Chapter~2]{SS96}, the authors also assume that $\Omega$ is connected. However, this assumption can be omitted by observing that any bounded open set $\Omega\subset\mathbb{R}^n$ with boundary $\partial\Omega$ of class $C^1$ has only a finite number of connected components. Therefore, it is sufficient to apply~\cite[Theorem~1.2, Chapter~2]{SS96} to each connected component to obtain Proposition~\ref{prop:priesz}. Furthermore, the estimate in~\eqref{eq:reg-est} follows from the one in~\cite[Theorem~1.2, Chapter~2]{SS96} by applying Poincaré's inequality.
\end{remark}

As a consequence of Proposition~\ref{prop:priesz}, we can argue as in~\cite[Theorem~3]{DFF20} to derive the following two local regularity results.

\begin{proposition}\label{prop:reg1}
Let $f\in W^{-1,p}_{\rm loc}(\Omega)$ and $u\in L^1_{\rm loc}(\Omega)$ be a solution to
\begin{equation*}
-\Delta u=f\quad\text{in $\mathcal{D}'(\Omega)$}.
\end{equation*}
Then $u\in W^{1,p}_{\rm loc}(\Omega)$ and for all bounded open sets $\Omega''\Subset\Omega'\Subset\Omega$ there exists a constant $C=C(n,p,\Omega',\Omega'')>0$ such that 
\begin{equation}\label{eq:wip-est}
\|u\|_{W^{1,p}(\Omega'')}\le C(\|u\|_{L^1(\Omega')}+\|f\|_{W^{-1,p}(\Omega')}).
\end{equation}
\end{proposition}

\begin{proof}
We fix two bounded open sets $\Omega''\Subset\Omega'\Subset\Omega$. Without loss of generality, we may assume that $\Omega'$ has a smooth boundary (otherwise, we replace $\Omega'$ with a bounded open set $\Omega'''$ with smooth boundary such that $\Omega''\Subset \Omega'''\Subset \Omega'$). To complete the proof, it is enough to show that $u\in W^{1,p}(\Omega'')$ and that the estimate~\eqref{eq:wip-est} holds.

Since $f \in W^{-1,p}(\Omega')$, by Proposition~\ref{prop:priesz} there is a function $v \in W^{1,p}_0(\Omega')$ satisfying
\begin{equation*}
-\Delta v=f\quad\text{in $\mathcal{D}'(\Omega')$},
\end{equation*}
and we can find a constant $C=C(n,p,\Omega')>0$ such that
\begin{equation} \label{eq:vf-est}
\|v\|_{W^{1,p}(\Omega')}\le C\|f\|_{W^{-1,p}(\Omega')}.
\end{equation}
The difference $u-v\in L^1(\Omega')$ is a harmonic distribution in $\Omega'$, so by Weyl's lemma (see, e.g.,~\cite[Lemma~1.16]{GM12}), there exists a harmonic function $z\in C^\infty(\Omega')$ such that $z=u-v$ in $\Omega'$. Therefore, $u=z+v\in W^{1,p}(\Omega'')$, and from~\eqref{eq:vf-est} it follows that
\begin{equation}\label{eq:reg1}
\|u\|_{W^{1,p}(\Omega'')}\le\|z\|_{W^{1,p}(\Omega'')}+\|v\|_{W^{1,p}(\Omega'')}\le\|z\|_{W^{1,p}(\Omega'')}+\|f\|_{W^{-1,p}(\Omega')}.
\end{equation}

Since $z$ is harmonic in $\Omega'$, Harnack's and Caccioppoli's inequalities (see, e.g.,~\cite[Theorem~7, Section~2.2]{Evans10}) imply the existence of a constant $C=C(n,p,\Omega',\Omega'')>0$ such that
\begin{align}
\|z\|_{W^{1,p}(\Omega'')}\le C\|z\|_{L^1(\Omega')}&\le C(\|u\|_{L^1(\Omega')}+\|v\|_{L^1(\Omega')}). \label{eq:reg2}
\end{align}
By combining~\eqref{eq:vf-est}--\eqref{eq:reg2}, we obtain~\eqref{eq:wip-est}.

\end{proof}

\begin{proposition}\label{prop:reg2}
Let $f\in L^p_{\rm loc}(\Omega)$ and $u\in L^1_{\rm loc}(\Omega)$ be a solution to
\begin{equation*}
-\Delta u=f\quad\text{in $\mathcal{D}'(\Omega)$}.
\end{equation*}
Then $u\in W^{2,p}_{\rm loc}(\Omega)$ and for all bounded open sets $\Omega''\Subset\Omega'\Subset\Omega$ there exists a constant $C=C(n,p,\Omega',\Omega'')>0$ such that 
\begin{equation}\label{eq:w2p-est}
\| u\|_{W^{2,p}(\Omega'')}\le C(\|u\|_{L^1(\Omega')}+\|f\|_{L^p(\Omega')}).
\end{equation}
\end{proposition}

\begin{proof}
We fix two bounded open sets $\Omega''\Subset\Omega'\Subset\Omega$. 

Since $f\in L^p_{\rm loc}(\Omega)\subset W^{-1,p}_{\rm loc}(\Omega)$, by Proposition~\ref{prop:reg1} we derive that $u\in W^{1,p}_{\rm loc}(\Omega)$ and that there exists a constant $C=C(n,p,\Omega',\Omega'')>0$ such that 
\begin{equation}\label{eq:reg3}
\|u\|_{W^{1,p}(\Omega'')}\le C(\|u\|_{L^1(\Omega')}+\|f\|_{W^{-1,p}(\Omega')})\le C(\|u\|_{L^1(\Omega')}+\|f\|_{L^p(\Omega')}).
\end{equation}
Moreover, $\nabla u\in L^p_{\rm loc}(\Omega;\mathbb{R}^n)$ is a distributional solution to 
\begin{equation*}
-\Delta(\nabla u)=\nabla f\quad\text{in $\mathcal{D}'(\Omega;\mathbb{R}^n)$},
\end{equation*}
with $\nabla f\in W^{-1,p}_{\rm loc}(\Omega;\mathbb{R}^n)$. By applying again Proposition~\ref{prop:reg1} we deduce that $\nabla u\in W^{1,p}_{\rm loc}(\Omega;\mathbb{R}^n)$ and that there exists a constant $C=C(n,p,\Omega',\Omega'')>0$ such that 
\begin{align}
\|\nabla u\|_{W^{1,p}(\Omega'';\mathbb{R}^n)}&\le C(\|\nabla u\|_{L^1(\Omega';\mathbb{R}^n)}+\|\nabla f\|_{W^{-1,p}(\Omega';\mathbb{R}^n)})\nonumber\\
&\le C(\|\nabla u\|_{L^1(\Omega';\mathbb{R}^n)}+\|f\|_{L^p(\Omega')}).\label{eq:reg4}
\end{align}
Hence, $u\in W^{2,p}_{\rm loc}(\Omega)$ and by combining~\eqref{eq:reg3} and~\eqref{eq:reg4} we derive~\eqref{eq:w2p-est}.
\end{proof}

We are now in a position to prove Theorem~\ref{thm:divcurl}, following the strategy adopted in~\cite[Theorem~4, Chapter~5]{Evans91} for the case $p=2$, see also~\cite[Théorème~2]{Murat78}.

\begin{proof}[Proof of Theorem~\ref{thm:divcurl}]

We fix a function $\varphi\in C_c^\infty(\Omega)$ and a bounded open set $\Omega'$ with smooth boundary such that  $\operatorname{supp}\varphi\subset\Omega'\Subset\Omega$.

For all $k\in\mathbb{N}$ we consider the function $V^k\in W^{1,p'}_0(\Omega';\mathbb{R}^n)$ given by Proposition~\ref{prop:priesz}, which solves
\begin{equation}\label{eq:uk-eq}
-\Delta V^k=G^k-G^\infty\quad\text{in $\mathcal{D}'(\Omega';\mathbb{R}^n)$}.
\end{equation}
In particular, by using also Remark~\ref{rem:negative-sob}, there exists a constant $C=C(n,p,\Omega)>0$, independent of $k\in\mathbb{N}$, such that for all $k\in\mathbb{N}$
\begin{equation*}
\|V^k\|_{W^{1,p'}(\Omega';\mathbb{R}^n)}\le C\|G^k-G^\infty\|_{W^{-1,p'}(\Omega';\mathbb{R}^n)}\le C\|G^k-G^\infty\|_{L^{p'}(\Omega';\mathbb{R}^n)}.
\end{equation*}
Therefore, the sequence $(V^k)_k\subset W^{1,p'}_0(\Omega';\mathbb{R}^n)$ is uniformly bounded, and by~\eqref{eq:weakconv} and~\eqref{eq:uk-eq} we derive that as $k\to\infty$
\begin{equation}\label{eq:uk-0}
V^k\to 0\quad\text{weakly in $W^{1,p'}_0(\Omega';\mathbb{R}^n)$}.
\end{equation}

Thanks to the local regularity result of Proposition~\ref{prop:reg2}, we derive that $u\in W^{2,p'}_{\rm loc}(\Omega';\mathbb{R}^n)$ and for all bounded open sets $\Omega''\Subset\Omega'$ we can find a constant $C=C(n,p,\Omega',\Omega'')>0$, independent of $k\in\mathbb N$, such that for all $k\in\mathbb{N}$
\begin{equation*}
\|V^k\|_{W^{2,p'}(\Omega'';\mathbb{R}^n)}\le C(\|V^k\|_{L^1(\Omega';\mathbb{R}^n)}+\|G^k-G^\infty\|_{L^{p'}(\Omega';\mathbb{R}^n)}).
\end{equation*}
In particular, the sequence $(V^k)_k\subset W^{2,p'}_{\rm loc}(\Omega';\mathbb{R}^n)$ is uniformly bounded. By~\eqref{eq:uk-0} and Rellich's theorem, we conclude that as $k\to\infty$
\begin{equation}\label{eq:uk-0-strong}
V^k\to 0\quad\text{weakly in $W^{2,p'}_{\rm loc}(\Omega';\mathbb{R}^n)$},\qquad
V^k\to 0\quad\text{strongly in $W^{1,p'}_{\rm loc}(\Omega';\mathbb{R}^n)$}.
\end{equation}

The function $\operatorname{curl} V^k\in W^{1,p'}_{\rm loc}(\Omega';\mathbb{R}^{n\times n})$ satisfies for all $k\in\mathbb{N}$
\begin{equation*}
-\Delta\operatorname{curl} V^k=\operatorname{curl} G^k-\operatorname{curl} G^\infty\quad\text{in $\mathcal{D}'(\Omega';\mathbb{R}^{n\times n})$}.
\end{equation*}
Therefore, by Proposition~\ref{prop:reg1} for all bounded open sets $\Omega'''\Subset\Omega''\Subset\Omega'$ there is $C=C(n,p,\Omega'',\Omega''')>0$, independent of $k\in\mathbb{N}$, such that for all $k\in\mathbb{N}$
\begin{align*}
&\|\operatorname{curl} V^k\|_{W^{1,p'}(\Omega''';\mathbb{R}^{n\times n})}\\
&\le C(\|\operatorname{curl} V^k\|_{L^1(\Omega'';\mathbb{R}^{n\times n})}+\|\operatorname{curl} G^k-\operatorname{curl} G^\infty\|_{W^{-1,p'}(\Omega'';\mathbb{R}^{n\times n})}).
\end{align*}
Hence, by~\eqref{eq:curlwk} and~\eqref{eq:uk-0-strong}, as $k\to\infty$ we have
\begin{equation}
\operatorname{curl} V^k\to 0\quad\text{strongly in $W^{1,p'}_{\rm loc}(\Omega';\mathbb{R}^{n\times n})$}.\label{eq:curluk-0}
\end{equation}

Let us define
\begin{equation}\label{eq:zk-def}
w^k\coloneq -\operatorname{div} V^k\quad\text{for all $k\in\mathbb{N}$},\qquad H^k\coloneq G^k-G^\infty-\nabla w^k\quad\text{for all $k\in\mathbb{N}$}.
\end{equation}
By~\eqref{eq:uk-0-strong} we have that $(w^k)_k\subset W^{1,p'}_{\rm loc}(\Omega'
)$ and as $k\to\infty$
\begin{equation}\label{eq:zk-conv}
w^k\to 0\quad\text{weakly in $W^{1,p'}_{\rm loc}(\Omega')$},\qquad w^k\to 0\quad\text{strongly in $L^{p'}_{\rm loc}(\Omega')$}.
\end{equation}
Moreover, $(H^k)_k\subset L^{p'}_{\rm loc}(\Omega';\mathbb{R}^n)$ and 
\begin{equation*}
H^k=-\operatorname{div}(\operatorname{curl} V^k)\quad\text{in $\mathcal{D}'(\Omega';\mathbb{R}^n)$ for all $k\in\mathbb{N}$}.
\end{equation*}
Indeed, for all $k\in\mathbb{N}$ and $i\in\{1,\dots,n\}$ we have
\begin{align*}
H_i^k&=G_i^k-G_i^\infty-\partial_iw^k=-\Delta V^k_i+\partial_i\operatorname{div} V^k\\
&=-\sum_{j=1}^n\partial_j(\partial_j V_i^k-\partial_i V_j^k)=-\sum_{j=1}^n\partial_j(\operatorname{curl} V^k)_{ij}=-(\operatorname{div}(\operatorname{curl} V^k))_i
\end{align*}
in $\mathcal{D}'(\Omega';\mathbb{R}^n)$. Therefore, by~\eqref{eq:curluk-0} we derive that as $k\to\infty$
\begin{equation}\label{eq:yk-0}
H^k\to 0\quad\text{strongly in $L^{p'}_{\rm loc}(\Omega';\mathbb{R}^n)$}.
\end{equation}

In view of~\eqref{eq:zk-def}, we can write
\begin{align*}
\int_\Omega\varphi F^k\cdot G^k\,{\mathrm d}x=\int_{\Omega'}\varphi F^k\cdot G^\infty\,{\mathrm d}x+\int_{\Omega'}\varphi F^k\cdot H^k\,{\mathrm d}x+\int_{\Omega'}\varphi F^k\cdot\nabla w^k\,{\mathrm d}x.
\end{align*}
Thanks to~\eqref{eq:weakconv},~\eqref{eq:divvk}, and~\eqref{eq:zk-conv}, we can apply Theorem~\ref{thm:simple-divcurl} to deduce that
\begin{align*}
\lim_{k\to\infty}\int_{\Omega'}\varphi F^k\cdot\nabla w^k\,{\mathrm d}x=0.
\end{align*}
Moreover, by~\eqref{eq:weakconv} and~\eqref{eq:yk-0} we get
\begin{align*}
\lim_{k\to\infty}\int_{\Omega'}\varphi F^k\cdot G^\infty\,{\mathrm d}x=\int_{\Omega'}\varphi F^\infty\cdot G^\infty\,{\mathrm d}x,\qquad\lim_{k\to\infty}\int_{\Omega'}\varphi F^k\cdot H^k\,{\mathrm d}x=0.
\end{align*}
Hence, for all $\varphi\in C_c^\infty(\Omega)$ we deduce
\begin{align*}
\lim_{k\to\infty}\int_\Omega\varphi F^k\cdot G^k\,{\mathrm d}x=\int_\Omega\varphi F^\infty\cdot G^\infty\,{\mathrm d}x,
\end{align*}
which gives~\eqref{eq:divcurl}.
\end{proof}

\section{Some Extensions to Fractional Sobolev Spaces}\label{sec:2}

In this section, we present several generalizations of the Div–Curl lemma to the fractional setting of~\cite{CS19,SS15}. We begin by briefly recalling the definitions of the Riesz fractional gradient and divergence, along with the key properties of the associated fractional Sobolev spaces. We then state and prove three generalizations of the Div–Curl lemma. The first, Theorem~\ref{thm:s-simple-divcurl}, extends Theorem~\ref{thm:simple-divcurl} to this fractional framework. The second, Theorem~\ref{thm:mixed1}, is a mixed local–nonlocal version, where we assume convergence of the fractional divergence for the first sequence and standard curl convergence for the second. The third, Theorem~\ref{thm:mixed2}, is another mixed local–nonlocal result, in the simplified setting of Theorem~\ref{thm:simple-divcurl}, where we replace the gradient with its fractional counterpart, while we keep the convergence of the standard divergence.

\subsection{The Fractional Framework}

In this subsection, we introduce the fractional operators that will be considered throughout this chapter and collect all the results that will be used in the proofs of the fractional Div-Curl lemmas.

Let $s\in (0,1)$ be fixed. We define
\begin{equation*}
\mu_{n,s}\coloneq\frac{2^s\Gamma(\frac{n+s+1}{2})}{\pi^\frac{n}{2}\Gamma(\frac{1-s}{2})}\in (0,\infty),
\end{equation*}
where $\Gamma$ denotes the Gamma function.

\begin{definition}
Given a function $\psi\in C_c^\infty(\mathbb{R}^n)$, we define the \emph{Riesz $s$-fractional gradient} $\nabla^s\psi\colon\mathbb{R}^n\to\mathbb{R}^n$ by
\begin{equation*}
\nabla^s\psi(x)\coloneq\mu_{n,s}\int_{\mathbb{R}^n}\frac{(\psi(y)-\psi(x))(y-x)}{|y-x|^{n+s+1}}\,{\rm d} y\quad\text{for all $x\in\mathbb{R}^n$}.
\end{equation*}
Given a function $\Psi\in C_c^\infty(\mathbb{R}^n;\mathbb{R}^n)$, we define the \emph{Riesz $s$-fractional divergence} $\operatorname{div}^s\Psi\colon\mathbb{R}^n\to\mathbb{R}$ by 
\begin{equation*}
\operatorname{div}^s\Psi(x)\coloneq\mu_{n,s}\int_{\mathbb{R}^n}\frac{(\Psi(y)-\Psi(x))\cdot(y-x)}{|y-x|^{n+s+1}}\,{\rm d} y\quad\text{for all $x\in\mathbb{R}^n$}.
\end{equation*}
\end{definition}

As observed in~\cite[Section 2]{CS19}, the nonlocal operators $\nabla^s\psi$ and $\operatorname{div}^s\Psi$ are well-defined in the sense that the above integrals converge for all $x\in\mathbb{R}^n$. Moreover, these operators satisfy the following $L^p$-type estimates, which can be found in~\cite[Lemmas~2.2-2.3 and Propositions~3.2-3.3]{CS23}.

\begin{proposition}\label{prop:nablas-est}
There exists a constant $C=C(n,s)>0$ such that for all $q\in [1,\infty]$ and for all $\psi\in C_c^\infty(\mathbb{R}^n)$ and $\Psi\in C_c^\infty(\mathbb{R}^n;\mathbb{R}^n)$ we have 
\begin{align*}
\|\nabla^s\psi\|_{L^q(\mathbb{R}^n;\mathbb{R}^n)}&\le C\|\psi\|_{L^q(\mathbb{R}^n)}^{1-s}\|\nabla\psi\|_{L^q(\mathbb{R}^n;\mathbb{R}^n)}^s,\\
\|\operatorname{div}^s\Psi\|_{L^q(\mathbb{R}^n)}&\le C\|\Psi\|_{L^q(\mathbb{R}^n;\mathbb{R}^n)}^{1-s}\|\nabla\Psi\|_{L^q(\mathbb{R}^n;\mathbb{R}^{n\times n})}^s.
\end{align*}
\end{proposition}

As shown in~\cite[Lemma~2.5]{CS19} and~\cite[Section 6]{Silhavy20}, the nonlocal operators $\nabla^s$ and $\operatorname{div}^s$ satisfy the following integration by parts formula.

\begin{proposition}\label{prop:dual}
For all $\psi\in C_c^\infty(\mathbb{R}^n)$ and $\Psi\in C_c^\infty(\mathbb{R}^n;\mathbb{R}^n)$ we have
\begin{equation}\label{eq:duality}
\int_{\mathbb{R}^n}\nabla^s\psi\cdot \Psi\,{\rm d} x = -\int_{\mathbb{R}^n}\psi\operatorname{div}^s\Psi \,{\rm d} x.
\end{equation}
\end{proposition}

\begin{remark}\label{rem:frac-lap}
These two fractional operators are closely connected to the classical {\em fractional Laplacian} $(-\Delta)^s$, which for $\psi\in C_c^\infty(\mathbb{R}^n)$ and $x\in\mathbb{R}^n$ is defined by
\begin{equation*}
(-\Delta)^s\psi(x)\coloneq
\begin{cases}
\displaystyle\nu_{n,s}\int_{\mathbb{R}^n}\frac{\psi(y)-\psi(x)}{|x-y|^{n+2s}}\,{\rm d}y &\text{if $s\in(0, \frac{1}{2})$},\\
\displaystyle \nu_{n,s}\lim_{\varepsilon\to 0^+}\int_{\{y\in\mathbb{R}^n\,:\,|x-y|>\varepsilon\}}\frac{\psi(y)-\psi(x)}{|x-y|^{n+2s}}\,{\rm d}y&\text{if $s\in[\frac{1}{2},1)$},
\end{cases}
\end{equation*}
where
\begin{equation*}
\nu_{n,s}\coloneq\frac{2^{2s}\Gamma(\frac{n+2s}{2})}{\pi^{\frac{n}{2}}\Gamma(-s)}\in(-\infty,0).
\end{equation*}
Indeed, as observed in~\cite[Theorem~5.3]{Silhavy20}, the following relation holds for all functions $\psi \in C_c^\infty(\mathbb{R}^n)$ and $s,r\in (0,1)$
\begin{equation}\label{eq:nablas-divs}
-\operatorname{div}^s\nabla^r\psi=(-\Delta)^{\frac{s+r}{2}}\psi\quad\text{in $\mathbb{R}^n$}.
\end{equation}
\end{remark}

We now introduce the fractional Sobolev framework for our Div-Curl lemmas. 

\begin{definition}
Given a function $\psi\in C_c^\infty(\mathbb{R}^n)$, we define the norm
\begin{equation*}
\|\psi\|_{H^{s,p}(\mathbb{R}^n)}\coloneq(\|\psi\|^p_{L^p(\mathbb{R}^n)} +\|\nabla^s\psi\|^p_{L^p(\mathbb{R}^n;\mathbb{R}^n)})^{1/p},
\end{equation*}
and the fractional Sobolev spaces $H^{s,p}_0(\Omega)$ and $H^{-s,p}(\Omega)$, respectively, by
\begin{equation*}
H_0^{s,p}(\Omega)\coloneq\overline{C_c^\infty(\Omega)}^{\|\,\cdot\,\|_{H^{s,p}(\mathbb{R}^n)}},\qquad H^{-s,p}(\Omega)\coloneq (H_0^{s,p'}(\Omega))'.
\end{equation*}
When $\Omega=\mathbb{R}^n$, we simply write $H^{s,p}(\mathbb{R}^n)\coloneq H_0^{s,p}(\mathbb{R}^n)$. Finally, we define the local negative fractional Sobolev space
\begin{equation*}
H^{-s,p}_{\rm loc}(\Omega)\coloneq\{f\in\mathcal{D}'(\Omega)\,:\, f\in H^{-s,p}(\Omega')\text{ for all open sets }\Omega'\Subset\Omega\}.
\end{equation*}
\end{definition}

\begin{remark}
We denote the above fractional Sobolev space by $H^{s,p}$ instead of $W^{s,p}$, as the latter is typically used for the fractional Sobolev-Slobodeckij space defined via Gagliardo seminorms. These two spaces generally differ, except in the case $p=2$, see~\cite[Theorems~1.7 and~2.2]{SS15}. In particular, when $p=2$, there exists a constant $C=C(n,s)>0$ such that for all $u\in H^{s,2}(\mathbb{R}^n)$ we have
\[
\|\nabla^s u\|_{L^2(\mathbb{R}^n;\mathbb{R}^n)}^2=C \int_{\mathbb{R}^n}\int_{\mathbb{R}^n}\frac{|u(y)-u(x)|^2}{|y-x|^{n+2s}}\,\mathrm{d}x\,\mathrm{d}y,
\]
see, for instance,~\cite[Proposition~2.8]{ABSS25}.
\end{remark}

If we extend the operators $\nabla^s$ and $\operatorname{div}^s$ to $H^{s,p}(\mathbb{R}^n)$ and $H^{s,p'}(\mathbb{R}^n;\mathbb{R}^n)$, respectively, then the integration by parts formula~\eqref{eq:duality} holds for all functions $\psi\in H^{s,p}(\mathbb{R}^n)$ and $\Psi\in H^{s,p'}(\mathbb{R}^n;\mathbb{R}^n)$. This justify the following definition.

\begin{definition}\label{def:divs}
Given a vector field $F\in L^p(\mathbb{R}^n;\mathbb{R}^n)$, we define $\operatorname{div}^s F\in H^{-s,p}_{\rm loc}(\Omega)$ for all bounded open sets $\Omega'\Subset\Omega$ by
\begin{equation*}
\langle\operatorname{div}^s F,v\rangle_{H^{-s,p}(\Omega')\times H^{s,p'}_0(\Omega)}\coloneq -\int_{\mathbb{R}^n}F\cdot\nabla^s v\,{\rm d}x\quad\text{for all $v\in H^{s,p'}_0(\Omega')$}.
\end{equation*}
\end{definition}

We recall the following result, which provide a useful connection between the Sobolev spaces $W^{1,p}(\mathbb{R}^n)$ and $H^{s,p}(\mathbb{R}^n)$. For the proof, we refer the interested reader to~\cite[Proposition 3.1]{KS22}.

\begin{proposition}\label{prop:equiv}
The following two statements hold.
\begin{enumerate}
\item[(a)] For all functions $v\in H^{s,p}(\mathbb{R}^n)$ there exists a function $u\in W^{1,p}_{\rm loc}(\mathbb{R}^n)$ such that
\begin{equation*}
\nabla u = \nabla^s v \quad \text{in $\mathbb{R}^n$}.
\end{equation*}
\item[(b)] For all functions $u\in W^{1,p}(\mathbb{R}^n)$ there exists a function $v\in H^{s,p}(\mathbb{R}^n)$ such that
\begin{equation*}
\nabla^s v =\nabla u\quad\text{in $\mathbb{R}^n$}.
\end{equation*}
Moreover, there exists a constant $C=C(n,s)>0$ such that
\begin{equation}\label{eq:lap-est}
\|v\|_{L^p(\mathbb{R}^n)} \le C \|u\|_{L^p(\mathbb{R}^n)}^{1-s}\|\nabla u\|_{L^p(\mathbb{R}^n;\mathbb{R}^n)}^s.
\end{equation}
\end{enumerate}
\end{proposition}

\begin{remark}
If we recall the fractional Laplacian defined in Remark~\ref{rem:frac-lap}, then in Proposition~\ref{prop:equiv}(b) we have $v=(-\Delta)^{\frac{1-s}{2}} u$. On the other hand, the function $v$ in Proposition~\ref{prop:equiv}(a) is constructed by using the inverse of the fractional Laplacian, which is the {\it Riesz potential}.
\end{remark}

We conclude this subsection by recalling a Leibniz-type rule for the fractional gradient. To this aim, we need to introduce the following nonlocal operator.

\begin{definition}
For all $\psi,\phi\in C_c^\infty(\mathbb{R}^n)$ we define $\nabla^s_{\rm NL}(\psi,\phi)\colon\mathbb{R}^n\to\mathbb{R}^n$ for all $x\in\mathbb{R}^n$ by
\begin{equation}\label{eq:nablaNL}
\nabla^s_{\rm NL}(\psi,\phi)(x)\coloneq\mu_{n,s}\int_{\mathbb{R}^n}\frac{(\psi(y)-\psi(x))(\phi(y)-\phi(x))(y-x)}{|y-x|^{n+s+1}}\,{\rm d} y
\end{equation}
\end{definition}

As before, the integral~\eqref{eq:nablaNL} converges for all $x\in\mathbb{R}^n$, and the following $L^p$-type estimate hold, see~\cite[Eq.~(2.11)]{KS22}.

\begin{proposition}
There exists a constant $C=C(n,s)>0$ such that for all functions $\psi,\phi\in C_c^\infty(\mathbb{R}^n)$ we have 
\begin{equation}\label{eq:NL}
\|\nabla^s_{\rm NL}(\psi,\phi)\|_{L^p(\mathbb{R}^n;\mathbb{R}^n)}\le C\|\psi\|_{L^p(\mathbb{R}^n)}\|\phi\|^s_{L^\infty(\mathbb{R}^n)}\|\nabla\phi\|^{1-s}_{L^\infty(\mathbb{R}^n;\mathbb{R}^n)}.
\end{equation}
\end{proposition}

In view of~\eqref{eq:NL}, the nonlocal operator $\nabla_{\rm NL}(\psi,\phi)$ can be extended to all pairs $(\psi,\phi)\in L^p(\mathbb{R}^n)\times C_c^\infty(\mathbb{R}^n)$. In particular, we have the following Leibniz-type rule for the fractional gradient, see~\cite[Lemma~2.11]{KS22} for a proof.

\begin{proposition}\label{prop:leibniz}
Let $\phi\in C_c^\infty(\Omega)$ and $u\in H^{s,p}(\mathbb{R}^n)$. Then, $u\phi\in H^{s,p}_0(\Omega)$ and
\begin{equation}\label{eq:leibniz}
\nabla^s(u\phi)(x)=u(x)\nabla^s\phi(x)+\phi(x)\nabla^s u(x)+\nabla^s_{\rm NL}(u,\phi)(x).
\end{equation}
Moreover, there exists a constant $C=C(n,s)>0$ such that 
\begin{equation}\label{eq:weak-NL}
\|\nabla^s_{\rm NL}(u,\phi)\|_{L^p(\mathbb{R}^n;\mathbb{R}^n)}\le C\|u\|_{L^p(\mathbb{R}^n)}\|\phi\|^s_{L^\infty(\mathbb{R}^n)}\|\nabla\phi\|^{1-s}_{L^\infty(\mathbb{R}^n;\mathbb{R}^n)}.
\end{equation}
\end{proposition}

For further properties of the Riesz fractional gradient and divergence, as well as of the fractional Sobolev space $H^{s,p}$, we refer the reader to~\cite{CS19,CS23,SS15,SS18,Silhavy20}.

\subsection{Three Fractional Div–Curl Lemmas}

We can finally present three possible extensions of the Div–Curl lemma to the fractional setting introduced earlier. We begin with a fractional version of Theorem~\ref{thm:simple-divcurl}.

\begin{theorem}[Simplified fractional Div-Curl lemma]\label{thm:s-simple-divcurl}
Let $(F^k)_k\subset L^p(\mathbb{R}^n;\mathbb{R}^n)$, $F^\infty\in L^p(\mathbb{R}^n;\mathbb{R}^n)$,  $(w^k)_k\subset  H^{s,p'}(\mathbb{R}^n)$, and $w^\infty \in H^{s,p'}(\mathbb{R}^n)$ satisfy as $k\to\infty$
\begin{align}
&F^k\to F^\infty& &\text{weakly in $L^p(\mathbb{R}^n;\mathbb{R}^n)$},\label{eq:s-simple-weakconv1}\\
&\nabla^s w^k\to \nabla^s w^\infty& &\text{weakly in $L^{p'}_{\rm loc}(\Omega;\mathbb{R}^n)$}.\label{eq:s-simple-weakconv2}
\end{align}
Assume that 
\begin{align}
&\operatorname{div}^s F^k\to\operatorname{div}^s F^\infty& &\text{strongly in $H^{-s,p'}_{\rm loc}(\Omega)$},\label{eq:s-simple-divvk1} \\
&w^k\to w^\infty& &\text{strongly in $L^{p'}(\mathbb{R}^n)$}.\label{eq:s-simple-divvk2}
\end{align}
Then, as $k\to\infty$
\begin{equation}
F^k\cdot\nabla^s w^k\to F^\infty\cdot\nabla^s w^\infty\quad\text{weakly* in $\mathcal{D}'(\Omega)$}.\label{eq:s-simple-divcurl}
\end{equation}
\end{theorem}

\begin{proof}
We fix a function $\varphi\in C_c^\infty(\Omega)$ and a bounded open set $\Omega'\subset\mathbb{R}^n$ with smooth boundary such that $\operatorname{supp}\varphi\subset\Omega'\Subset\Omega$.

By Proposition~\ref{prop:leibniz} we have
\begin{align}
\int_{\mathbb{R}^n}\varphi F^k\cdot\nabla^s w^k\,{\rm d}x&=\int_{\mathbb{R}^n} F^k\cdot\nabla^s(\varphi w^k)\,{\rm d}x-\int_{\mathbb{R}^n} F^k\cdot\nabla^s\varphi  w^k\,{\rm d}x\nonumber\\
&\quad-\int_{\mathbb{R}^n} F^k\cdot\nabla^s_{\rm NL}(w^k,\varphi)\,{\rm d}x.\label{eq:s1}
\end{align}
Since $\nabla^s\varphi\in L^\infty(\mathbb{R}^n;\mathbb{R}^n)$ by Proposition~\ref{prop:nablas-est}, the weak convergence~\eqref{eq:s-simple-weakconv1} and the strong convergence in~\eqref{eq:s-simple-divvk2} imply
\begin{equation}\label{eq:s2}
\lim_{k\to\infty}\int_{\mathbb{R}^n} F^k\cdot\nabla^s\varphi  w^k\,{\rm d}x=\int_{\mathbb{R}^n} F^\infty\cdot\nabla^s\varphi  w^\infty\,{\rm d}x.
\end{equation}
By the estimate~\eqref{eq:weak-NL}, the strong convergence in~\eqref{eq:s-simple-divvk2}, and the fact that $\nabla^s_{\rm NL}(\,\cdot\,,\,\cdot\,)$ is a bilinear operator, we deduce that as $k\to\infty$
\begin{equation*}
\nabla^s_{\rm NL}(w^k,\varphi)\to\nabla^s_{\rm NL}(w^\infty,\varphi)\quad\text{strongly in $L^{p'}(\mathbb{R}^n;\mathbb{R}^n)$}.
\end{equation*}
Thus, in view of~\eqref{eq:s-simple-weakconv1}, we conclude that
\begin{equation}\label{eq:s3}
\lim_{k\to\infty}\int_{\mathbb{R}^n} F^k\cdot\nabla^s_{\rm NL}(w^k,\varphi)\,{\rm d}x=\int_{\mathbb{R}^n} F^\infty\cdot\nabla^s_{\rm NL}(w^\infty,\varphi)\,{\rm d}x.
\end{equation}

Finally, by Proposition~\ref{prop:leibniz} we have that $\varphi w^k\in H^{s,p'}_0(\Omega')$. Moreover, in view of~\eqref{eq:leibniz} and~\eqref{eq:s-simple-weakconv2}, as $k\to\infty$ we derive
\begin{equation*}
\varphi w^k\to\varphi w^\infty\quad\text{weakly in $H^{s,p'}_0(\Omega')$}.
\end{equation*}
Hence, by using also~\eqref{eq:s-simple-divvk1}, we have
\begin{align}
\lim_{k\to\infty}\int_{\mathbb{R}^n} F^k\cdot\nabla^s(\varphi w^k)\,{\rm d}x&=-\lim_{k\to\infty}\langle\operatorname{div}^sF^k,\varphi w^k\rangle_{H^{-s,p}(\Omega')\times H^{s,p'}_0(\Omega')}\nonumber\\
&=-\langle\operatorname{div}^sF^\infty,\varphi w^\infty\rangle_{H^{-s,p}(\Omega')\times H^{s,p'}_0(\Omega')}\nonumber\\
&=\int_{\mathbb{R}^n} F^\infty\cdot\nabla^s(\varphi w^\infty)\,{\rm d}x.\label{eq:s4}
\end{align}
By combining~\eqref{eq:s1}--\eqref{eq:s4}, for all $\varphi\in C_c^\infty(\Omega)$ we obtain 
\begin{align*}
\lim_{k\to\infty}\int_{\mathbb{R}^n}\varphi F^k\cdot\nabla^s w^k\,{\rm d}x&=
\int_{\mathbb{R}^n} F^\infty\cdot\nabla^s(\varphi w^\infty)\,{\rm d}x-\int_{\mathbb{R}^n} F^\infty\cdot\nabla^s\varphi  w^\infty\,{\rm d}x\\
&\quad-\int_{\mathbb{R}^n} F^\infty\cdot\nabla^s_{\rm NL}(w^\infty,\varphi)\,{\rm d}x\\
&=\int_{\mathbb{R}^n}\varphi  F^\infty\cdot\nabla^sw^\infty\,{\rm d}x,
\end{align*}
which is~\eqref{eq:s-simple-divcurl}.
\end{proof}

\begin{remark}\label{rem:conv}
Let us highlight the main differences compared to the analogous result stated in Theorem~\ref{thm:simple-divcurl}.

\begin{enumerate}
\item[(a)] The first difference is that the sequences $(F^k)_k$ and $(w^k)_k$ must be defined on the whole space $\mathbb{R}^n$, rather than just on $\Omega$. This requirement arises from the nonlocal nature of the fractional gradient: in order to define $\nabla^s \psi$ and $\operatorname{div}^s\Psi$ at a point $x\in\mathbb{R}^n$, one needs to know the values of $\psi$ and $\Psi$ on the entire space $\mathbb{R}^n$.

\item[(b)] Moreover, the weak convergence~\eqref{eq:s-simple-weakconv1} must also hold on the whole space $\mathbb{R}^n$. This is a consequence of the integration by parts formula stated in Proposition~\ref{prop:dual}, which involves integration over all of $\mathbb{R}^n$. In particular, the integrals in~\eqref{eq:s4} must be over $\mathbb{R}^n$, which in turn forces all the integrals in~\eqref{eq:s1} to be taken over $\mathbb{R}^n$ as well. Consequently, the sequence $(F^k)_k$ must converge to $F^\infty$ weakly in $L^p(\mathbb{R}^n;\mathbb{R}^n)$.

\item[(c)] A further difference is that we must explicitly assume~\eqref{eq:s-simple-divvk2}. In the local case, such strong convergence is only required on $\Omega'\Subset \Omega$, a bounded open set, and therefore follows from Rellich’s theorem. In the nonlocal case, however, the strong convergence must hold globally on $\mathbb{R}^n$. We observe that, if $\Omega \subset \mathbb{R}^n$ is a bounded open set and $(w^k)_k \subset H_0^{s,p'}(\Omega)$, $w^\infty \in H_0^{s,p'}(\Omega)$ satisfy as $k \to \infty$
\begin{equation*}
\nabla^s w^k \to \nabla^sw^\infty \quad \text{weakly in } L^{p'}(\mathbb{R}^n;\mathbb{R}^n),
\end{equation*}
then~\eqref{eq:s-simple-divvk2} follows directly from the fractional versions of the Poincaré inequality (see~\cite[Theorem~3.3]{SS15}) and Rellich’s compactness theorem (see~\cite[Theorem~2.2]{SS18} and~\cite[Theorem~2.3]{BCMC20}).
\end{enumerate}
\end{remark}

\begin{remark}
We point out to the interested reader that, in the context of fractional operators defined via Gagliardo-type double integrals, an analogue of Theorem~\ref{thm:s-simple-divcurl} can be found in~\cite[Lemma~3.3]{FBRS17}. In particular, this version is used to prove the compactness of a nonlinear Gagliardo-type differential operator depending on an oscillating scalar weight.
\end{remark}

As a consequence of Theorem~\ref{thm:s-simple-divcurl} and Proposition~\ref{prop:equiv}, we can derive the following “mixed local–nonlocal” Div–Curl lemma. In this result, the fractional divergence is assumed to converge in $H^{-s,p}_{\rm loc}(\Omega)$, while the classical curl is required to converge in $W^{-1,p'}_{\rm loc}(\Omega; \mathbb{R}^{n \times n})$.

\begin{theorem}[First mixed local-nonlocal Div-Curl lemma]\label{thm:mixed1}
Let $(F^k)_k\subset L^p(\mathbb{R}^n;\mathbb{R}^n)$, $F^\infty\in L^p(\mathbb{R}^n;\mathbb{R}^n)$,  $(G^k)_k\subset  L^{p'}_{\rm loc}(\Omega;\mathbb{R}^n)$, and $G^\infty \in L^{p'}_{\rm loc}(\Omega;\mathbb{R}^n)$ satisfy as $k\to\infty$
\begin{align}\label{eq:mixed-weakconv}
F^k\to F^\infty\quad\text{weakly in $L^p(\mathbb{R}^n;\mathbb{R}^n)$},\qquad 
G^k\to G^\infty\quad\text{weakly in $L^{p'}_{\rm loc}(\Omega;\mathbb{R}^n)$}.
\end{align}
Assume that as $k\to\infty$
\begin{align*}
\operatorname{div}^sF^k&\to\operatorname{div}^sF^\infty & &\text{strongly in $H^{-s,p}_{\rm loc}(\Omega)$},\\
\operatorname{curl}G^k&\to\operatorname{curl}G^\infty& &\text{strongly in $W^{-1,p'}_{\rm loc}(\Omega;\mathbb{R}^{n\times n})$}.
\end{align*}
Then, as $k\to\infty$
\begin{equation}\label{eq:mixed-divcurl}
F^k\cdot G^k\to F^\infty\cdot G^\infty\quad\text{weakly* in $\mathcal{D}'(\Omega)$}.
\end{equation}
\end{theorem}

\begin{proof}
We fix a function $\varphi\in C_c^\infty(\Omega)$ and a bounded open set $\Omega'\subset\mathbb{R}^n$ with smooth boundary such that
$\operatorname{supp}\varphi\subset\Omega'\Subset\Omega$. We also fix a function $\psi\in C_c^\infty(\Omega')$ such that $\psi=1$ on $\operatorname{supp}\varphi$. 

Let $(w^k)_k\subset W^{1,p'}_{\rm loc}(\Omega')$ and $(H^k)_k\subset L^{p'}_{\rm loc}(\Omega';\mathbb{R}^n)$ be the functions defined in~\eqref{eq:zk-def}. Since $\psi=1$ on $\operatorname{supp}\varphi$, by~\eqref{eq:zk-def} we have
\begin{align}
\int_\Omega\varphi F^k\cdot G^k\,{\rm d}x&=\int_\Omega\varphi F^k\cdot H^k\,{\rm d}x+\int_\Omega\varphi F^k\cdot G^\infty\,{\rm d}x+\int_\Omega\varphi F^k\cdot\nabla w^k\,{\rm d}x\nonumber\\
&=\int_{\mathbb{R}^n}\varphi F^k\cdot H^k\,{\rm d}x+\int_{\mathbb{R}^n}\varphi F^k\cdot G^\infty\,{\rm d}x+\int_{\mathbb{R}^n}\varphi F^k\cdot\nabla (\psi w^k)\,{\rm d}x.\label{eq:mixed1}
\end{align}
Thanks to~\eqref{eq:yk-0} and~\eqref{eq:mixed-weakconv}, we derive
\begin{equation}\label{eq:mixed2}
\lim_{k\to\infty}\int_{\mathbb{R}^n}\varphi F^k\cdot H^k\,{\rm d}x=0.
\end{equation}
Moreover, by using again~\eqref{eq:mixed-weakconv} we have
\begin{equation}\label{eq:mixed3}
\lim_{k\to\infty}\int_{\mathbb{R}^n}\varphi F^k\cdot G^\infty\,{\rm d}x=\int_{\mathbb{R}^n}\varphi F^\infty\cdot G^\infty\,{\rm d}x=\int_\Omega\varphi F^\infty\cdot G^\infty\,{\rm d}x.
\end{equation}

It remains to study the last term in~\eqref{eq:mixed1}. We have that $(\psi w^k)_k\subset W^{1,p}_0(\Omega')$ and
\begin{equation*}
\psi w^k\to 0\quad\text{weakly in $W^{1,p}_0(\Omega')$ and strongly in $L^{p'}(\Omega')$},
\end{equation*}
in view of~\eqref{eq:zk-conv}. Therefore, by Proposition~\ref{prop:equiv}(b), there exists $(v^k)_k\subset H^{s,p}(\mathbb{R}^n)$ such that 
\begin{equation*}
\nabla^s v^k=\nabla (\psi w^k)\quad\text{in $\mathbb{R}^n$},
\end{equation*}
and by exploiting the estimate~\eqref{eq:lap-est} we derive
\begin{equation*}
v^k\to 0\quad\text{weakly in $H^{s,p'}(\mathbb{R}^n)$ and strongly in $L^{p'}(\mathbb{R}^n)$}.
\end{equation*}
Therefore, we can apply Theorem~\ref{thm:s-simple-divcurl} to derive 
\begin{equation}\label{eq:mixed4}
\lim_{k\to\infty}\int_{\mathbb{R}^n}\varphi F^k\cdot\nabla (\psi w^k)\,{\rm d}x=\lim_{k\to\infty}\int_{\mathbb{R}^n}\varphi F^k\cdot\nabla^s v^k\,{\rm d}x=0.
\end{equation}
In view of~\eqref{eq:mixed1}--\eqref{eq:mixed4} we finally get~\eqref{eq:mixed-divcurl}.
\end{proof}

The last version of the fractional Div–Curl lemma we present is a mixed formulation, stated in the simplified setting of Theorem~\ref{thm:simple-divcurl}, in which the classical gradient in the product is replaced by a fractional gradient.

\begin{theorem}[Second mixed local-nonlocal Div-Curl lemma]\label{thm:mixed2}
Let $(F^k)_k\!\subset\! L^p_{\rm loc}(\Omega;\mathbb{R}^n)$, $F^\infty\in L^p_{\rm loc}(\Omega;\mathbb{R}^n)$, $(w^k)_k\subset  H^{s,p'}(\mathbb{R}^n)$, and $w^\infty \in H^{s,p'}(\mathbb{R}^n)$ satisfy as $k\to\infty$
\begin{align}
&F^k\to F^\infty& &\text{weakly in $L^p_{\rm loc}(\Omega;\mathbb{R}^n)$},\label{eq:mixed-weakconv-21}\\
&\nabla^s w^k\to \nabla^s w^\infty& &\text{weakly in $L^{p'}_{\rm loc}(\Omega;\mathbb{R}^n)$}.\label{eq:mixed-weakconv-22}
\end{align}
Assume that as $k\to\infty$
\begin{equation}\label{eq:mixed-divvk-2}
\operatorname{div} F^k\to\operatorname{div} F^\infty \quad\text{strongly in $W^{-1,p}_{\rm loc}(\Omega)$}.
\end{equation}
Then, as $k\to\infty$
\begin{equation}\label{eq:mixed-divcurl-2}
F^k\cdot \nabla^s w^k\to F^\infty\cdot \nabla^s w^\infty\quad\text{weakly* in $\mathcal{D}'(\Omega)$}.
\end{equation}
\end{theorem}

\begin{proof}
We fix a bounded open set $\Omega'\Subset\Omega$ with smooth boundary. 

Let $(u^k)_k\in W^{1,p}_{\rm loc}(\mathbb{R}^n)$ and $u^\infty\in W^{1,p}_{\rm loc}(\mathbb{R}^n)$ be the functions given by Proposition~\ref{prop:equiv}(a) associated to $(w^k)_k\subset H^{s,p'}(\mathbb{R}^n)$ and $w^\infty\in H^{s,p'}(\mathbb{R}^n)$, respectively. By Poincare's inequality, there exists a constant $C=C(n,p,\Omega')>0$, independent of $k\in\mathbb{N}$, and a sequence $(c^k)_k\subset\mathbb{R}$ satisfying for all $k\in\mathbb N$
\begin{equation*}
\|u^k-c^k\|_{L^p(\Omega')}\le C\|\nabla u^k\|_{L^p(\Omega';\mathbb{R}^n)}=C\|\nabla^s w^k\|_{L^p(\Omega';\mathbb{R}^n)}\le C.
\end{equation*}
Hence, we can find a subsequence $(k_h)_h\subset \mathbb N$ and function $z^\infty\in W^{1,p}(\Omega')$ such that as $h\to\infty$
\begin{equation}\label{eq:weak-gk}
z^{k_h}\coloneq u^{k_h}-c^{k_h}\to z^\infty\quad\text{weakly in $W^{1,p}(\Omega')$}.
\end{equation}
By~\eqref{eq:mixed-weakconv-22}, for all $\Psi\in C_c^\infty(\Omega';\mathbb{R}^n)$ we have
\begin{align*}
\int_{\Omega'} \nabla z^\infty\cdot\Psi\,{\rm d}x&=\lim_{h\to\infty}\int_{\Omega'} \nabla z^{k_h}\cdot\Psi\,{\rm d}x\\
&=\lim_{h\to\infty}\int_{\Omega'} \nabla u^{k_h}\cdot\Psi\,{\rm d}x=\lim_{h\to\infty}\int_{\Omega'} \nabla^s w^{k_h}\cdot\Psi\,{\rm d}x=\int_{\Omega'}\nabla^s w^\infty\cdot\Psi\,{\rm d}x.
\end{align*}
Hence, $\nabla w^\infty=\nabla^s w^\infty$ in $\Omega'$. Thus in view of~\eqref{eq:mixed-weakconv-21},~\eqref{eq:mixed-divvk-2}, and~\eqref{eq:weak-gk} we can apply Theorem~\ref{thm:simple-divcurl} and obtain that as $h\to\infty$
\begin{align*}
F^{k_h}\cdot \nabla^s w^{k_h}= F^{k_h}\cdot \nabla z^{k_h} \to F^\infty\cdot \nabla z^\infty= F^\infty\cdot \nabla^s w^\infty\quad\text{weakly* in $\mathcal{D}'(\Omega')$}.
\end{align*}
Since this holds for every subsequence, we derive~\eqref{eq:mixed-divcurl-2} in $\Omega'$. By the arbitrariness of $\Omega'\Subset\Omega$, the conclusion follows.
\end{proof}

\begin{remark}
We point out that these two versions of the mixed local-nonlocal Div–Curl lemma, although partial (since they do not involve the fractional curl), still have useful applications. For instance, although not stated explicitly, such results are used in the proof of the compactness of fractional linear operators with oscillating coefficients in~\cite[Theorem~3.1]{CCM24}.

It is clear that, by exploiting the definition of $\nabla^s$, one can also introduce a fractional curl operator, following Definition~\ref{def:curl}. Therefore, it would be interesting to extend these fractional Div–Curl lemmas to the case where the second sequence has a strongly converging fractional curl. The main issues are twofold. On the one hand, one must find suitable elliptic regularity estimates in order to derive a fractional Helmholtz decomposition. On the other hand, as explained in Remark~\ref{rem:conv}, the strong convergence~\eqref{eq:s-simple-divvk2} is required to hold on the entire space $\mathbb{R}^n$, which is not straightforward to obtain from elliptic regularity results, as these are typically local.
\end{remark}

%---------------------------
%   Acknowledgement
%---------------------------

\begin{acknowledgement}
The author wishes to thank Alessandro Carbotti and Alberto Maione for many helpful discussions on the topic.
\end{acknowledgement}

%---------------------------
%   Funding Information
%---------------------------

\ethics{Funding Information}{The author is a member of the Gruppo Nazionale per l'Analisi Matematica, la Probabilit\`a e le loro Applicazioni (INdAM-GNAMPA), and acknowledges the support of the the INdAM-GNAMPA 2025 Project ``DISCOVERIES - Difetti e Interfacce in Sistemi Continui: un’Ottica Variazionale in Elasticità con Risultati Innovativi ed Efficaci Sviluppi'' (CUP: E5324001950001). The author acknowledges also the support of the Centre de Recerca Matemàtica of Barcelona (CRM), under the International Programme for Research in Groups (IP4RG). Finally, the author has been founded by the European Union-NextGenerationEU under the Italian Ministry of University and Research (MUR) National Centre for HPC, Big Data and Quantum Computing (CN\_00000013 – CUP: E13C22001000006).}

%---------------------------
%   Competing Interests
%---------------------------

\ethics{Competing Interests}{The author has no conflicts of interest to declare that are relevant to the content of this chapter.}

%---------------------------
%   Appendix
%---------------------------

%\section*{Appendix}

\end{document}